\documentclass[12pt]{article}

\usepackage{amsmath}

\usepackage{times}
\usepackage{bm}
\usepackage{natbib}
\usepackage{amsfonts}
\usepackage{graphicx}
\usepackage{comment}
\usepackage{mathrsfs}
\usepackage{color}
\usepackage{amsthm}
\usepackage{comment}
\usepackage{url}

\newcommand{\blind}{0}

\newtheorem{proposition}{Proposition}

\begin{document}

\def\spacingset#1{\renewcommand{\baselinestretch}%
{#1}\small\normalsize} \spacingset{1}


\if0\blind
{
  \title{\bf Sliced rotated sphere packing designs}
  \author{Xu He\thanks{He's work is partial supported by Special National Key Research and Development Plan under Grant No. 2016YFD0400206, 
National Natural Science Foundation of China (NSFC 11501550 and NSFC 11671386) and funding from Chinese Ministry of Science and Technology (Grant No. 2016YFF0203801).
}\hspace{.2cm}\\
    Academy of Mathematics and System Sciences, \\Chinese Academy of Sciences}
  \maketitle
} \fi

\if1\blind
{
  \bigskip
  \bigskip
  \bigskip
  \begin{center}
    {\LARGE\bf Sliced rotated sphere packing designs}
\end{center}
  \medskip
} \fi

\bigskip
\begin{abstract}
Space-filling designs are popular choices for computer experiments. 
A sliced design is a design that can be partitioned into several subdesigns. 
We propose a new type of sliced space-filling design called sliced rotated sphere packing designs. 
Their full designs and subdesigns are rotated sphere packing designs. 
They are constructed by rescaling, rotating, translating and extracting the points from a sliced lattice. 
We provide two fast algorithms to generate such designs. 
Furthermore, we propose a strategy to use sliced rotated sphere packing designs adaptively. 
Under this strategy, initial runs are uniformly distributed in the design space, follow-up runs are added by incorporating information gained from initial runs, and the combined design is space-filling for any local region. 
Examples are given to illustrate its potential application. 
\end{abstract}

\noindent%
\vfill
\hfill {\tiny technometrics tex template (do not remove)}

\newpage
\spacingset{1.45} 
\section{Introduction}
\label{sec:intro}

Space-filling designs whose points are uniformly scattered in the design space are popular choices for computer experiments~\citep{Santner:book,Sacks:1989}. 
In this work, we consider designs which contain points in $[0,1]^p$. 
The separation distance of a design $\mathbf{D}$ is 
\begin{equation} \label{eqn:dp}
 \min_{\mathbf{x}_1,\mathbf{x}_2\in \mathbf{D}, \mathbf{x}_1\neq \mathbf{x}_2} \left( \|\mathbf{x}_1-\mathbf{x}_2\|_2 \right),   
\end{equation}
and the fill distance of a design $\mathbf{D}$ is 
\begin{equation*}\label{eqn:df}
 \sup_{\mathbf{z}\in [0,1]^p}\left\{ \min_{\mathbf{x}\in \mathbf{D}}(\|\mathbf{z}-\mathbf{x}\|_2) \right\}. 
\end{equation*}
As discussed in~\citet{Johnson:1990} and~\citet{Haaland:2017}, designs with high separation distance or low fill distance have some optimal or asymptotically optimal characteristics that are desirable for computer experiments. 
Many space-filling designs are Latin hypercube designs, which achieve optimal one-dimensional projective uniformity~\citep{McKay:1979}. 
Maximin distance Latin hypercube designs~\citep{Morris:1995}, which are generated by numerically maximizing the separation distance within the class of Latin hypercube design, are popular among space-filling designs. 

Lattice-based designs are another type of space-filling designs~\citep{Heitmann:2016,RSPD,ILmMD}. 
A lattice is the collection of infinitely many orthogonal or nonorthogonal grid points, 
and a lattice-based design consists of the lattice points that are located in the design space. 
Lattice-based designs have identical structure at any local area and are therefore space-filling globally and locally. 
In particular, rotated sphere packing designs are constructed by rescaling, rotating and translating the lattice that has asymptotically lowest fill distance~\citep{RSPD}. 

If a space-filling design can be partitioned into several space-filling subdesigns, we call the full design, together with the slicing rule, a sliced space-filling design~\citep{Qian:2009:SSFD}.  
\citet{Qian:2012} proposed sliced Latin hypercube designs whose full designs and subdesigns are Latin hypercube designs. 
Later, \citet{Ba:2015} proposed optimal sliced Latin hypercube designs whose full designs and subdesigns are Latin hypercube designs with high separation distances. 
For illustration, an optimal sliced Latin hypercube design in two dimensions is presented in Figure~\ref{fig:intro}(a). 
Other types of sliced space-filling designs are proposed by~\citet{Qian:2009:SSFD,Yang:2013,Ai:2014,Huang:2014,Sun:2014,Xie:2014,Deng:2015,Liu:2015,SOALH}, among others. 
These designs are useful for 
computer experiments with quantitative and qualitative variables~\citep{Qian:Wu:Wu:2008,Deng:2016}, 
computer experiments with multiple levels of accuracy~\citep{Qian:Wu:2008}, 
and model validation~\citep{Zhang:2013}. 

\begin{figure}
\begin{center}
\includegraphics[width=13.6cm]{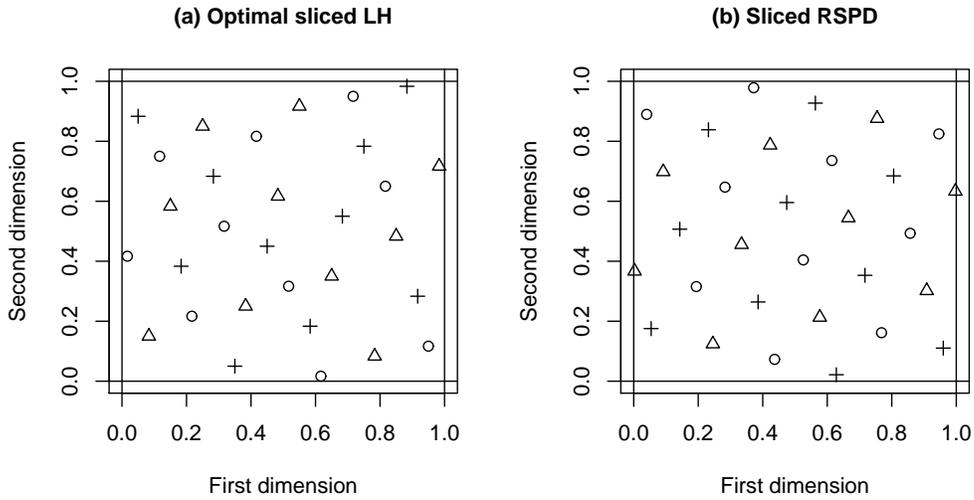}
\end{center}
\caption{Two sliced space-filling designs in two dimensions. Different slices are depicted by different symbols.  
\label{fig:intro}}
\end{figure}
%

In this paper, we propose a new class of sliced space-filling design called sliced rotated sphere packing designs. 
Their full designs and subdesigns are rotated sphere packing designs. 
Sliced rotated sphere packing designs are constructed based on sliced lattices, which are lattices that can be partitioned into several sublattices. 
An example of sliced rotated sphere packing design in two dimensions is displayed in Fig.~\ref{fig:intro}(b). 
We provide two algorithms to construct sliced rotated sphere packing designs. 
The first algorithm partitions an ordinary rotated sphere packing design and the second algorithm enlarges an ordinary rotated sphere packing design. 
Both algorithms are simple without any numerical steps. 
The full design and subdesigns of a sliced rotated sphere packing design achieve the same degree of uniformity as ordinary rotated sphere packing designs that are based on the same type of lattice. 
While any type of lattice can be sliced, we provide a space-filling type of sliced lattice, based on which sliced rotated sphere packing designs achieve better separation distances by (\ref{eqn:dp}) than optimal sliced Latin hypercube designs for low-dimensional cases. 

Sliced rotated sphere packing designs are also useful as sequential or adaptive designs. 
Usually, computer experiments are carried out sequentially. 
Many sequential space-filling designs distribute points uniformly in the design space for both initial and follow-up runs 
~\citep{Qian:2009,Qian:Ai:2010,NOALH,Xu:2015,Kong:2016}.  
However, as we gain more knowledge on the input-output system, we may find some regions more interesting for further investigation than others. 
This requires adaptive designs that can incorporate information gained from completed computer runs. 
One such example is sequential minimum energy designs~\citep{Joseph:energy} whose points are representative of a probability density. 
By assigning higher density for more interesting areas, we can obtain nonuniform designs that focus on critical areas. 
Such designs are adaptive if the density function is set based on information gained from completed computer runs. 

Similar to the idea of sequential minimum energy designs, 
we propose a strategy to use sliced rotated sphere packing designs adaptively. 
Under this strategy, design points are generated with two densities,  
a higher density for more interesting regions and a lower density for the remaining space, with the high density region chosen adaptively based on information gained from initial runs that are uniformly distributed in the design space. 
Unlike most adaptive design methods that search for optimal next-point-to-add from the whole design space, 
we propose to search over a short list of candidate follow-up runs. 
Because of separation distance properties among all initial and candidate follow-up runs, 
the generated points are space-filling globally for low density regions and locally for high density regions. 
Furthermore, adjacent points of sliced rotated sphere packing designs are connected with special rules which may simplify the definition of high density region. 
As a result, this strategy is useful for emulation of nonstationary computer experiments and optimization of computer experiments.

The rest of the paper is organized as follows: 
Section~\ref{sec:lattice} gives preliminary mathematical results on lattices and sliced lattices. 
In Section~\ref{sec:const}, we give the algorithms to construct sliced rotated sphere packing designs. 
In Section~\ref{sec:simu}, we compare sliced rotated sphere packing designs with other classes of sliced designs numerically. 
In Section~\ref{sec:adaptive}, we give the strategy to use sliced rotated sphere packing designs adaptively and show its applications.  
Conclusions and discussion are provided in Section~\ref{sec:conclusion}. 
Proofs are given in the appendix. 

\section{Lattices and sliced lattices} \label{sec:lattice}

In this section, we give necessary definitions and results of lattices and sliced lattices. 

A set of points in $\mathbb{R}^p$ is called a lattice if it forms a group. 
The lattice points consist of linear combinations of $p$ basis vectors with integer coefficients. 
We call a $p\times p$ matrix a generator matrix of the lattice if its rows are the basis vectors.
As an example, the set of integer vectors, $\mathbb{Z}^p$, is called the $p$-dimensional integer lattice, which can be generated from the $p$-dimensional identity matrix. 
Two important properties of lattices are their densities and thicknesses. 
If we place identical balls in $\mathbb{R}^p$ centered at the lattice points, then the maximum radius of the balls such that no two balls overlap is called the packing radius of the lattice, 
and the minimum radius of the balls such that the union of overlapped balls cover $\mathbb{R}^p$ is called the covering radius of the lattice.  
The Voronoi cell of a point $\mathbf{x}_0$ in a lattice $\mathbf{L}$ is the region 
\[ \text{Vor}(\mathbf{x}_0) = \{ \mathbf{z}: |\mathbf{z}-\mathbf{x}_0| \leq |\mathbf{z}-\mathbf{x}|, \mbox{ for any } \mathbf{x} \in \mathbf{L} \}. \]
The density and thickness of a lattice is the volume of one ball with packing and covering radius, respectively, divided by the volume of one Voronoi cell. 
Lattice-based designs with highest possible density and lowest possible thickness have asymptotically optimal separation distance and fill distance, respectively. 

In this paper, we focus on two types of lattices, $A_p$ and $A_p^*$. 
The $A_p$ is called the $p$-dimensional zero-sum root lattice, with one possible generator matrix 
\begin{eqnarray}\label{eqn:A_p}
\mathbf{M}_p &=& \frac{\sqrt{2}}{2} \mathbf{I}_p - \frac{\sqrt{p+1}+1}{\sqrt{2}p} \mathbf{J}_p,  
\end{eqnarray}
where $\mathbf{I}_p$ is the $p\times p$ identity matrix and $\mathbf{J}_p$ is the $p\times p$ matrix with all entries being one. 
The $A_p^*$ is called the dual of the $p$-dimensional zero-sum root lattice, with one generator matrix
\begin{eqnarray}\label{eqn:A_p^*}
\mathbf{M}_p^* &=&\frac{\sqrt{p+1}}{\sqrt{p}} \mathbf{I}_p - \frac{1}{\sqrt{p}(\sqrt{p+1}-1)} \mathbf{J}_p. 
\end{eqnarray}
The $A_p$ and $A_p^*$ are equivalent when $p=2$. 
Their densities and thicknesses for $2\leq p\leq 10$ are given in Tables~\ref{tab:density} and~\ref{tab:thickness}, respectively. 
The $A_p^*$ has the best known thickness for $2\leq p\leq 22$ and the $A_p$ has the best known density for $p=2$ and $3$. 
\citet{RSPD} recommended to use $A_p^*$ for constructing rotated sphere packing designs, but as can be seen from the table, both lattices are substantially more space-filling than $\mathbb{Z}^p$. 
For a comprehensive review of lattices, see \citet{Conway:1998} or~\citet{Zong:1999}. 

\begin{table}
\caption{Density of three lattices, $A_p$, $A_p^*$ and $\mathbb{Z}^p$}
\begin{center}
\begin{tabular}{cccccccccc}
$p$ & 2 & 3 & 4 & 5 & 6 & 7 & 8 & 9 & 10 \\
\hline
$A_p$   & 0.907 & 0.740 & 0.552 & 0.380 & 0.244 & 0.148 & 0.085 & 0.046 & 0.024 \\
$A_p^*$ & 0.907 & 0.680 & 0.441 & 0.255 & 0.135 & 0.065 & 0.030 & 0.013 & 0.005 \\
$\mathbb{Z}^p$   & 0.785 & 0.524 & 0.308 & 0.164 & 0.081 & 0.037 & 0.016 & 0.006 & 0.002 \\
\end{tabular}\label{tab:density}
\end{center}
\end{table}

\begin{table}
\caption{Thickness of three lattices, $A_p$, $A_p^*$ and $\mathbb{Z}^p$}
\begin{center}
\begin{tabular}{cccccccccc}
$p$ & 2 & 3 & 4 & 5 & 6 & 7 & 8 & 9 & 10 \\
\hline
$A_p$   & 1.21 & 2.09 & 3.18 & 5.92 & 9.84 & 18.9 & 33.0 & 64.4 & 116.0 \\
$A_p^*$ & 1.21 & 1.46 & 1.77 & 2.12 & 2.55 & 3.06 & 3.67 & 4.39 & 5.25 \\
$\mathbb{Z}^p$   & 1.57 & 2.72 & 4.93 & 9.20 & 17.4 & 33.5 & 64.9 & 126.8 & 249.0 \\
\end{tabular}\label{tab:thickness}
\end{center}
\end{table}

Next, we give our definition and some theoretical results for sliced lattices. 
Suppose $\mathbf{K}$ is a proper subgroup of a lattice $\mathbf{L}$. 
For any $\mathbf{q} \in \mathbf{L}$, the set $\mathbf{K} \oplus \mathbf{q} = \{ \mathbf{k}+\mathbf{q} : \mathbf{k}\in \mathbf{K}\}$ is called a coset of $\mathbf{K}$. 
If $\mathbf{Q}$ is a finite subset of $\mathbf{L}$ and any coset of $\mathbf{K}$ can be uniquely expressed by $\mathbf{K} \oplus \mathbf{q}$ with a $\mathbf{q} \in \mathbf{Q}$, 
then the cosets of $\mathbf{K}$ partition $\mathbf{L}$ and we call $(\mathbf{L},\mathbf{K},\mathbf{Q})$ a sliced lattice. 

It is not hard to see that any lattice $\mathbf{L}$ can be partitioned into sublattices. 
Suppose $\mathbf{G}$ is a generator matrix of $\mathbf{L}$ and $z$ is an integer greater than one.
Let $\mathbf{K}$ denote the lattice generated from $z \mathbf{G}$, then $(\mathbf{L},\mathbf{K},\{0,\ldots,z-1\}^p)$ is a sliced lattice with $z^p$ slices. 
Despite its simplicity, $(\mathbf{L},\mathbf{K},\{0,\ldots,z-1\}^p)$ is not useful for medium to large $p$ due to its large number of slices. 
A more practical type of sliced lattice is $(A_p^*,A_p,\mathbf{B})$, 
where $\mathbf{B}=\{\mathbf{u}_0,\ldots,\mathbf{u}_p\}$ and $\mathbf{u}_j$ is the $p$-vector with the first $j$ elements being one and other elements being zero. 
Here $\mathbf{u}_0$ is the zero vector. 
Proposition~\ref{prp:slice} below shows that $(A_p^*,A_p,\mathbf{B})$ is a sliced lattice with $p+1$ slices. 

\begin{proposition} \label{prp:slice}
Let $\mathbf{L}$ and $\mathbf{K}$ be the lattices generated by $\mathbf{M}_p^*$ in (\ref{eqn:A_p^*}) and $\{2(p+1)\}^{1/2}\mathbf{M}_p$ in (\ref{eqn:A_p}), respectively. 
Suppose $\mathbf{a}^T \mathbf{M}_p^* \in \mathbf{L}$ where $\mathbf{a}=(a_1,\ldots,a_p)^T$ is an integer vector with $\sum a_i \mod (p+1) =z$,  
then $\mathbf{a}^T \mathbf{M}_p^* \in \mathbf{K} \oplus \mathbf{u}_z$.
Furthermore, $\{ \mathbf{K} \oplus \mathbf{u}_0, \ldots, \mathbf{K} \oplus \mathbf{u}_p \}$ is a partition of $\mathbf{L}$. 
\end{proposition}

For a sliced lattice $(\mathbf{L},\mathbf{K},\mathbf{Q})$, think of $\mathbf{L}$ as being enlarged from $\mathbf{K}$. 
We call points in $\mathbf{K}$ ``adult'' points and the remaining points of $\mathbf{L}$ ``baby'' points. 
For a given baby point, the adult points nearest to it are called the ``parents'' of the baby point and the baby point is called a ``child'' of its parents. 
As shall be shown in Section~\ref{sec:adaptive}, the parent-child relation is useful for adaptive designs. 
Proposition~\ref{prp:parent} below gives the parent-child relation of the $(A_p^*,A_p,\mathbf{B})$ sliced lattice. 

\begin{proposition} \label{prp:parent}
(i) Suppose $(\mathbf{L},\mathbf{K},\mathbf{B})$ is a sliced lattice where $\mathbf{L}$ and $\mathbf{K}$ are generated by $\mathbf{M}_p^*$ in (\ref{eqn:A_p^*}) and $\{2(p+1)\}^{1/2}\mathbf{M}_p$ in (\ref{eqn:A_p}), respectively. 
For any baby point $(b_1,\ldots,b_p) \mathbf{M}_p^* \in \mathbf{L}\setminus \mathbf{K}$ with $\sum b_i \mod (p+1) =z$, 
its parents are $\{(b_1,\ldots,b_p) - (c_1,\ldots,c_p)\} \mathbf{M}_p^* $, 
of which either $( c_i\in\{0,1\}, i=1,\ldots p, \sum c_i = z )$ or $( c_i \in \{0,-1\}, i=1,\ldots p, \sum c_i = z-(p+1) )$. 

(ii) Conversely, any adult point $(a_1,\ldots,a_p) \mathbf{M}_p^* \in \mathbf{K}$ has $2^{p+1}-2$ children, which can be written as $\{(a_1,\ldots,a_p) - (c_1,\ldots,c_p)\} \mathbf{M}_p^*$ with either $( c_i \in \{0,1\}, i=1,\ldots p, (c_1,\ldots,c_p)\neq 0 )$ or $( c_i \in \{0,-1\}, i=1,\ldots p, (c_1,\ldots,c_p)\neq 0 )$. 
\end{proposition}

\section{Construction} \label{sec:const}

\subsection{Construction of rotated sphere packing designs} \label{sec:const:RSPD}

Before proposing our algorithms to construct sliced rotated sphere packing designs, we first give a brief review of the construction of rotated sphere packing designs proposed in~\citet{RSPD}.  
A rotated sphere packing design is a finite set of points generated from rescaling, rotating, translating and extracting the points from a lattice. 
With $p$, $n$ and the generator matrix $\mathbf{G}$ given, the algorithm has five major steps: 
\begin{enumerate}
\item Obtain a rotation matrix $\mathbf{R}$, which is a $p\times p$ orthogonal matrix. 
\item Obtain a large design given by $\mathbf{E} = \mathbf{F} \mathbf{G} \mathbf{R}$, 
where $\mathbf{F}$ is an integer matrix sufficiently large such that 
$\{ \mathbf{f}^T\mathbf{G}\mathbf{R} : \mathbf{f} \in \mathbb{Z}^p, \mathbf{f}^T\mathbf{G}\mathbf{R} \in [-l/2-\rho_c,l/2+\rho_c]^p \}$ is a subset of rows of $\mathbf{E}$, 
where   
$ l = \left(n \Omega_p/\Theta\right)^{1/p} \rho_c$, 
$\Omega_p$ is the volume of one unit sphere in $\mathbb{R}^p$, $\Theta$ is the thickness of the lattice and $\rho_c$ is the covering radius of the lattice. 
\item Search for a perturbation vector $\boldsymbol{\delta} =(\delta_1,\ldots,\delta_p)^T \in \text{Vor}(0)$ such that 
there are exactly $n$ points of $\mathbf{E}$ contained in the region $\otimes_{k=1}^p [-l/2-\delta_k,l/2-\delta_k]$, 
where $\text{Vor}(0)$ is the Voronoi cell of $(0,\ldots,0) \in \mathbf{E}$. 
A theorem in~\citet{RSPD} guarantees the existence of such $\boldsymbol{\delta}$. 
\item Obtain the design $\mathbf{D}$ by extracting points of $ \mathbf{\tilde E}/l+1/2 $ that lie in $[0,1]^p$,  
where $\mathbf{\tilde E}$ is the matrix obtained by adding $\boldsymbol{\delta}^T$ to rows of $\mathbf{E}$. 
\item Repeat Steps~1-4 for $w$ times and select the $(\mathbf{R},\boldsymbol{\delta})$ combination that maximizes the empirical projected uniformity measured by the criterion~\citep{Roshan:2015} 
\begin{equation}
\label{eqn:criterion:MaxPro}
\psi(\mathbf{D}) = \left\{ \{n(n-1)\}^{-1}
\sum_{1\leq i<j\leq n} \frac{1}{\prod_{k=1}^p (x_{i,k}-x_{j,k})^2} \right\}^{1/p}. 
\end{equation} 
\end{enumerate}

In~\citet{RSPD}, $\mathbf{G}$ is recommended to be $\mathbf{M}_p^*$ in (\ref{eqn:A_p^*}). 
A Givens rotation $\mathbf{R}_p(i,j,\alpha)$ is the $p\times p$ identity matrix with the $(i,i)$th, $(i,j)$th, $(j,i)$th and $(j,j)$th elements being replaced by $\cos(\alpha)$, $-\sin(\alpha)$, $\sin(\alpha)$ and $\cos(\alpha)$, respectively. 
For $p=2$, $\mathbf{R}=\mathbf{I}_2$ and $w=1$ was recommended.
For $p>2$, it was recommended to use $w=100$ and generate $\mathbf{R}$s randomly by multiplying $p(p-1)/2$ sequential Givens rotations with $\alpha$ sampled independently and uniformly from $[0,2\pi]$.

\subsection{Construction of sliced rotated sphere packing designs} \label{sec:const:SRSPD}

We now give two algorithms to construct sliced rotated sphere packing designs based on a sliced lattice $(\mathbf{L},\mathbf{K},\mathbf{Q})$ and an ordinary rotated sphere packing design. 
The algorithms are general for any type of sliced lattice.  
Let $\mathbf{G}$ and $\mathbf{H}$ be the generator matrices of $\mathbf{L}$ and $\mathbf{K}$, respectively, and assume $\mathbf{Q}=\{\mathbf{q}_1,\ldots,\mathbf{q}_s\}$ with $\mathbf{q}_1=0$. 
For the $(A_p^*,A_p,\mathbf{B})$, we can use $\mathbf{G}=\mathbf{M}_p^*$ in (\ref{eqn:A_p^*}), $\mathbf{K}=\{2(p+1)\}^{1/2}\mathbf{M}_p$ in (\ref{eqn:A_p}) and $s=p+1$.  

The first algorithm partitions an $\mathbf{L}$-based rotated sphere packing design with the following three steps: 
\begin{enumerate}
\item Obtain $\mathbf{D}$, an ordinary $\mathbf{G}$-based rotated sphere packing design with $n$ points as in Section~\ref{sec:const:RSPD}, 
and express its points as $\mathbf{x}_i=(\mathbf{a}_i^T \mathbf{G} \mathbf{R} + \boldsymbol{\delta}^T)/l + 1/2$, $\mathbf{a}_i \in \mathbb{Z}^p$, $i=1,\ldots,n$. 
\item Determine the coset $\mathbf{a}_i^T \mathbf{G}$ belongs to, $i=1,\ldots,n$. 
\item Obtain $\mathbf{D}_k = \{\mathbf{x}_i: \mathbf{a}_i^T \mathbf{G} \in \mathbf{K} \oplus \mathbf{q}_k\}$, $k=1,\ldots,s$. 
\end{enumerate}
The sliced rotated sphere packing design is given by $(\mathbf{D}_1,\ldots,\mathbf{D}_s)$. 
This algorithm is suitable for simultaneous construction of sliced rotated sphere packing designs.  

The second algorithm enlarges a $\mathbf{K}$-based rotated sphere packing design with the following five steps: 
\begin{enumerate}
\item Obtain $\mathbf{D}_1$, an ordinary $\mathbf{H}$-based rotated sphere packing design with $n_1$ points as in Section~\ref{sec:const:RSPD},
and express its points as $\mathbf{x}_i=(\mathbf{a}_i^T \mathbf{H} \mathbf{R} + \boldsymbol{\delta}^T)/l + 1/2$, $\mathbf{a}_i \in \mathbb{Z}^p$, $i=1,\ldots,n_1$. 
\item Obtain a large design given by $\mathbf{E} = \mathbf{F} \mathbf{G} \mathbf{R}$, 
where $\mathbf{F}$ is an integer matrix sufficiently large such that 
$\{ \mathbf{f}^T\mathbf{G}\mathbf{R} : \mathbf{f} \in \mathbb{Z}^p, \mathbf{f}^T\mathbf{G}\mathbf{R} \in [-l/2-\rho_c,l/2+\rho_c]^p \}$ is a subset of rows of $\mathbf{E}$.
\item Obtain the design $\mathbf{D}$ by extracting points of $ \mathbf{\tilde E}/l+1/2 $ that lie in $[0,1]^p$,  
where $\mathbf{\tilde E}$ is the matrix obtained by adding $\boldsymbol{\delta}^T$ to rows of $\mathbf{E}$. 
\item Determine the coset $\mathbf{a}_i^T \mathbf{G}$ belongs to, $i=1,\ldots,n$. 
\item Obtain $\mathbf{D}_k = \{\mathbf{x}_i: \mathbf{a}_i^T \mathbf{G} \in \mathbf{K} \oplus \mathbf{q}_k\}$, $k=2,\ldots,s$. 
\end{enumerate}
The sliced rotated sphere packing design is given by $(\mathbf{D}_1,\ldots,\mathbf{D}_s)$. 
This algorithm is appealing for sequential experiments in which the adult points are given by $\mathbf{D}_1$ and the baby points are given by $\cup_{k=2}^s \mathbf{D}_k$. 

From both algorithms, the desired sliceable structure comes with no lose of uniformity and little extra computation. 
The full design achieves the same degree of uniformity as an ordinary $\mathbf{L}$-based rotated sphere packing design while the subdesigns achieve the same degree of uniformity as ordinary $\mathbf{K}$-based rotated sphere packing designs. 
Although the algorithms are applicable to arbitrary sliced lattices, 
the resulted sliced design is space-filling only if its underlying sliced lattice is space-filling. 
As a result, in this paper we focus on $(A_p^*,A_p,\mathbf{B})$-based sliced rotated sphere packing designs. 
We can define the parent-child relation of sliced rotated sphere packing designs similarly to that for sliced lattices. 
We illustrate the parent-child relation of an $(A_2^*,A_2,\mathbf{B})$-based sliced rotated sphere packing design in Figure~\ref{fig:child}. 

\begin{figure}
\begin{center}
\includegraphics[width=13.6cm]{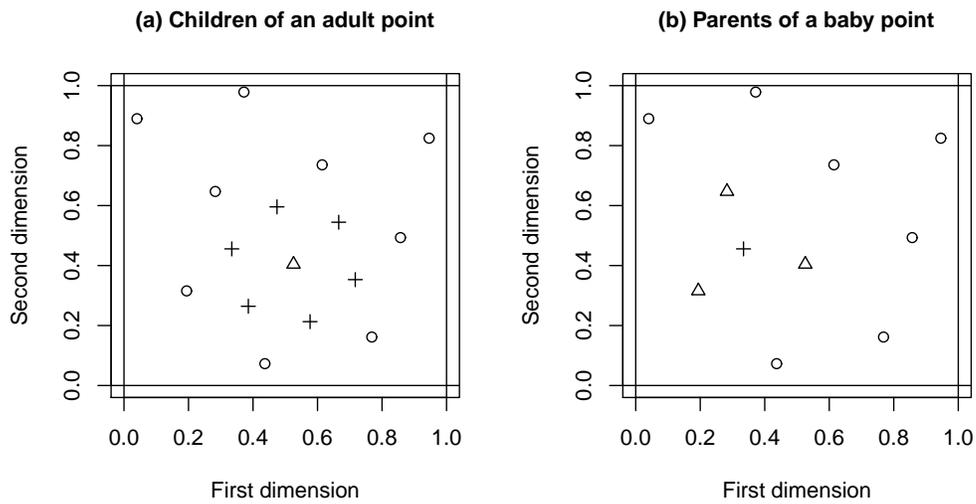}
\end{center}
\caption[width=10cm]{The parent-child relation of an $(A_2^*,A_2,\mathbf{B})$-based sliced rotated sphere packing design. 
The panel on the left depicts the six children (pluses) of an adult point (triangle) and other adult points (circles). 
The panel on the right depicts the three parents (triangles) of a baby point (plus) and other adult points (circles). 
\label{fig:child}}
\end{figure}

Let $n_j$ denote the number of points for $\mathbf{D}_j$ and $n=\sum n_j$,  
the first algorithm allows $n$ to be pre-specified and the second algorithm allows $n_1$ to be pre-specified. 
However, neither algorithm allows us to simultaneously set the values of $(n_1,\ldots,n_s)$. 
This is a major limitation of sliced rotated sphere packing designs. 
In some applications the balance property that $n_1=\cdots=n_s$ may be important. 
The balance property of a sliced design $\mathbf{D}$ can be measured by the criterion
\begin{equation*} \label{eqn:sn}
 \varphi(\mathbf{D}) = \sum_{j=1}^{s} (n_j-n/s)^2.
\end{equation*}
Recall that in the construction algorithm of rotated sphere packing designs, we generate a number of $(\mathbf{R},\boldsymbol{\delta})$ randomly. 
Figure~\ref{fig:dnpsi} plots $\varphi(\mathbf{D})$ and $\psi(\mathbf{D})$ in (\ref{eqn:criterion:MaxPro}) of 100 random designs when $p=4$ and $n=50$. 
It is observed that the $n_1,\ldots,n_s$ are likely to be roughly equal but not exactly the same. 
Ten of these designs achieve minimum $\varphi(\mathbf{D})$, i.e., 2. 
If balance is the primary concern, we can choose the design with minimum $\psi(\mathbf{D})$ among those 10 designs. 
This design is almost balanced and has $n_1=9$, $n_2=11$ and $n_3=n_4=n_5=10$. 
Meanwhile, it also achieves good projection uniformity. 
We can further reduce $\varphi(\mathbf{D})$ by searching for $\boldsymbol{\delta}$ around $(\mathbf{R},\boldsymbol{\delta})$ combinations that yield low $\varphi(\mathbf{D})$, but we omit the details here.
From our experience, it is not hard to find a balanced $(A_p^*,A_p,\mathbf{B})$-based sliced rotated sphere packing design for $p\leq 5$ with $n\leq 200$.  
However, as $p$ and $n$ grow, it becomes much harder to find strictly balanced designs. 

\begin{figure}
\begin{center}
\includegraphics[width=10cm]{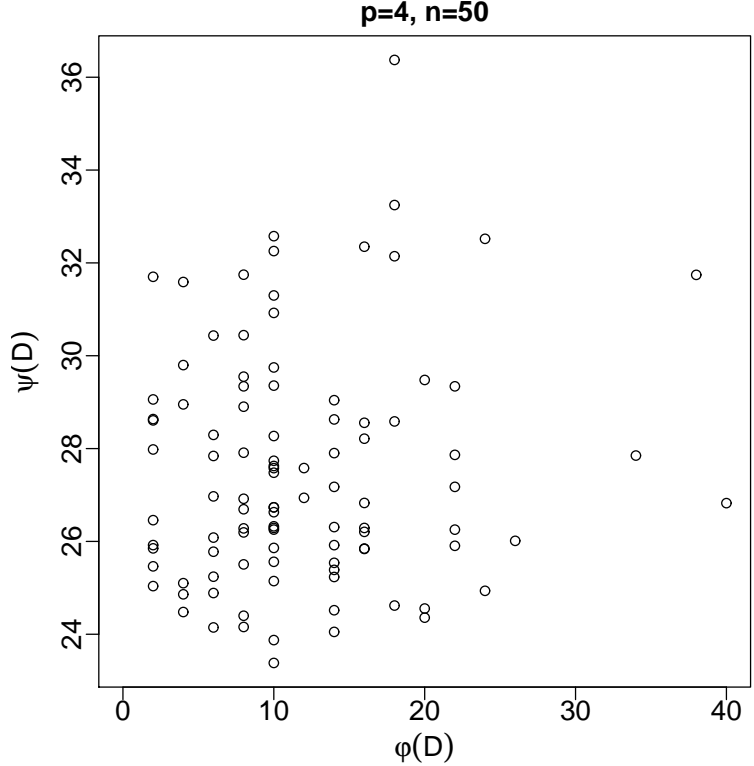}
\end{center}
\caption[width=10cm]{The $\varphi(\mathbf{D})$ and $\psi(\mathbf{D})$ from 100 randomly generated designs, $p=4$, $n=50$. 
\label{fig:dnpsi}}
\end{figure}

\section{Numerical comparison on separation distance}\label{sec:simu}

In this section, we compare $(A_p^*,A_p,\mathbf{B})$-based sliced rotated sphere packing designs with optimal sliced Latin hypercube designs and sliced Latin hypercube designs using the separation distance criterion by (\ref{eqn:dp}). 
As discussed, this criterion reflects the uniformity of designs and is what optimal sliced Latin hypercube designs aim to maximize. 
For $(A_p^*,A_p,\mathbf{B})$-based sliced rotated sphere packing designs generated from the first algorithm,  
the separation distance for the full design is $p^{1/2} (p+1)^{(1-p)/(2p)} n^{-1/p}$, 
the same to that of an ordinary $A_p^*$-based rotated sphere packing design, 
and the separation distance among points in the same slice is $\sqrt{2} (p+1)^{1/(2p)} n^{-1/p}$,   
the same to that of an ordinary $A_p$-based rotated sphere packing design with $n/(p+1)$ points.  

We obtain separation distance of the other two types of sliced designs numerically. 
The comparison for $2\leq p\leq 10$, $n=10(p+1)$ and $n=40(p+1)$ are shown in Figure~\ref{fig:maximin}. 
Seen from the results, 
for $n=10(p+1)$, sliced rotated sphere packing design is the best for $2\leq p\leq 6$; 
for $n=40(p+1)$, sliced rotated sphere packing design is the best for $2\leq p\leq 7$. 
These results were expected by us, since ordinary rotated sphere packing designs have better separation distance than maximin distance Latin hypercube designs for $2\leq p\leq 6$~\citep{RSPD}, 
sliced rotated sphere packing designs retain the same separation distances as ordinary rotated sphere packing designs, and optimal sliced Latin hypercube designs have inferior separation distances to maximin distance Latin hypercube designs. 
From our experience, sliced rotated sphere packing design is more competitive as $n$ grows. 
Besides, the construction of sliced rotated sphere packing designs is fast for $2\leq p\leq 6$. 
For instance, it takes 106 seconds to generate a sliced rotated sphere packing design with $p=6$ and $n=600$ on a laptop. 
To sum it, sliced rotated sphere packing designs have good distance-based properties and should be useful for applications such as 
computer experiments with quantitative and qualitative variables~\citep{Qian:Wu:Wu:2008,Deng:2016}, 
computer experiments with multiple levels of accuracy~\citep{Qian:Wu:2008}, 
and model validation~\citep{Zhang:2013}. 

\begin{figure}
\begin{center}
\includegraphics[width=13.6cm]{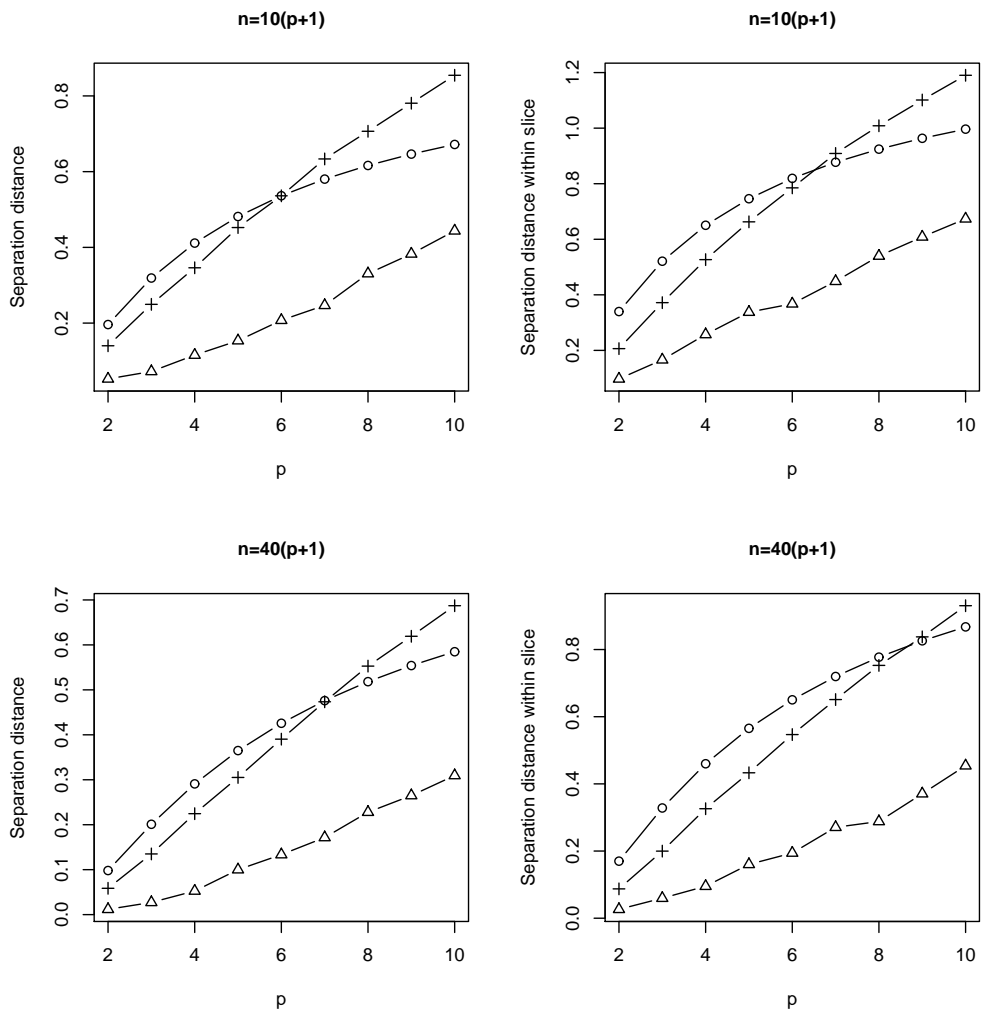}
\end{center}
\caption{Separation distance of all points (left) and points in the same slice (right) for three types of sliced designs: sliced rotated sphere packing designs (circles), optimal sliced Latin hypercube designs (pluses) and sliced Latin hypercube designs (triangles).   
\label{fig:maximin}}
\end{figure}

\section{Adaptive sliced rotated sphere packing designs} \label{sec:adaptive}


In this section, we propose a strategy to use sliced rotated sphere packing designs adaptively. 
Adaptive designs are widely used in many computer experiment problems including 
emulation of nonstationary computer experiments~\citep{Jin:2002}, 
global optimization of black-box functions~\citep{EI}, 
finding several promising points for response surface optimum~\citep{Joseph:energy}, 
finding an excursion set whose output is above a target value~\citep{Chevalier:2014} and   
estimating a percentile of the output distribution~\citep{Oakley:2004}.
Here we focus on the emulation and optimization objectives. 

\begin{figure}
\begin{center}
\includegraphics[width=10cm]{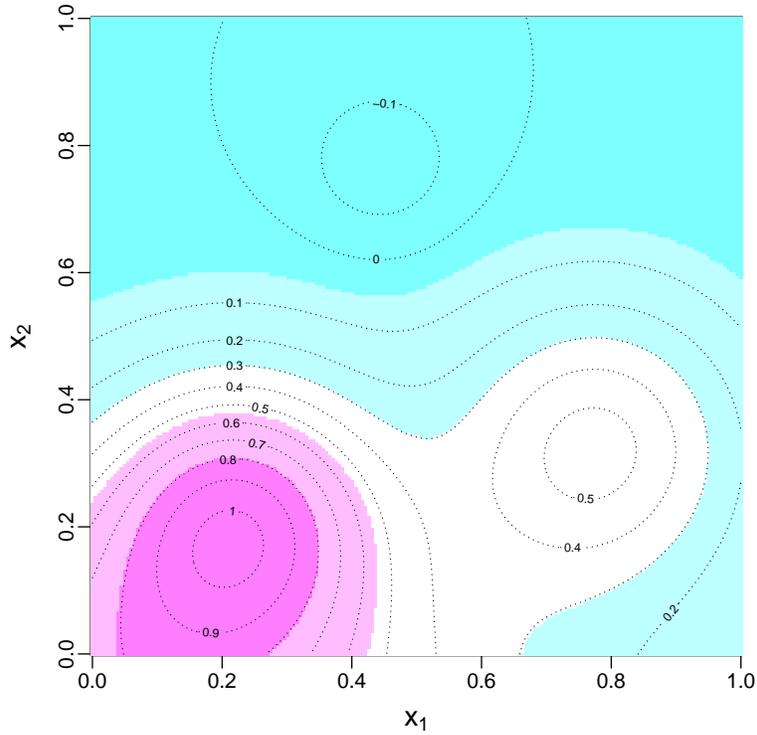}
\end{center}
\caption[width=10cm]{Counter plot of the Franke's function. 
\label{fig:resShow}}
\end{figure}

Many computer experiments are nonstationary. 
For example, Figure~\ref{fig:resShow} gives the contour plot of the Franke's function~\citep{Joseph:energy}. 
It can be seen that the output has larger volatility in the bottom-left corner than in other places. 
Thus, in order to obtain an overall accurate emulator, a space-filling design with denser points in the bottom-left corner is desired. 
The major challenge here is to identify the high volatility region based on limited computer runs. 
Although Gaussian process emulators can give variance estimates for any position in the design space, 
a stationary Gaussian process model will not yield high variance estimate for high volatility regions. 
One notable nonstationary Gaissuan process model is the treed Gaussian process model which partitions the design space based on volatility and fit different Gaussian process models separately~\citep{Gramacy:2008}. 
An adaptive design approach using treed Gaussian process model was proposed in~\citet{Gramacy:2009}. 
However, from our experience, this approach does not work well for small sample sizes. 
Another cross-validation based approach was proposed in~\citet{Jin:2002} which has three major steps below:  
\begin{enumerate}
\item Carry out initial runs that come from a maximin distance Latin hypercube design. 
\item Fit Gaussian process emulators using completed runs and add follow-up runs one-by-one. 
Let the cross-validation error be defined by 
\begin{equation}~\label{eqn:e}
 e(\mathbf{x}_{n+1}) = \left[ \left\{ \sum_{i=1}^n \left( \hat f_{-i}(\mathbf{x}_{n+1}) - \hat f(\mathbf{x}_{n+1}) \right)^2 \right\}/n \right]^{1/2}, 
\end{equation}
where $n$ is the number of completed runs, $\hat f(\mathbf{x}_{n+1})$ is the predicted outcome from emulating all completed runs and $\hat f_{-i}(\mathbf{x}_{n+1})$ is the predicted outcome without using the $i$th run. 
Let 
\begin{equation}~\label{eqn:g}
 g(\mathbf{x}_{n+1}) = e(\mathbf{x}_{n+1}) \text{min}_{i=1}^n \|\mathbf{x}_{n+1}-\mathbf{x}_i\|. 
\end{equation}
The new point $\mathbf{x}_{n+1}$ shall maximize $g(\mathbf{x}_{n+1})$ where $\mathbf{x}_{n+1}\in [0,1]^p$. 
\item Stop when a certain number of points are added or $\sup_{\mathbf{x}} g(\mathbf{x})$ goes below a value.  
\end{enumerate}

In this algorithm, high $e(\mathbf{x})$ implies high volatility around $\mathbf{x}$ and high $\text{min}_{i=1}^n \|\mathbf{x}-\mathbf{x}_i\|$ implies good interpoint distance. 
As a result, the added runs tend to locate in high volatility regions and not too close to any completed run. 

Finding the response surface minimum is another important objective for computer experiments. 
A related objective is to find several promising points for response surface minimum. 
The promising points can be further investigated by extra experiments based on same or different responses.
For this objective, sequential minimum energy designs~\citep{Joseph:energy} are suitable which has three major steps below:  
\begin{enumerate}
\item Carry out initial runs that come from a maximin distance Latin hypercube design. 
\item Fit a Gaussian process emulator using completed runs and add follow-up runs one-by-one. 
Let the density function be defined by 
\begin{equation}~\label{eqn:d}
 d(\mathbf{x}) = \hat f_{\text{max}} - \hat f(\mathbf{x}), 
\end{equation}
where $\hat f_{\text{max}}$ is the estimated global maximum output value and $\hat f(\mathbf{x})$ is the predicted outcome value at $\mathbf{x}$. 
Let the energy function be 
\begin{equation} \label{eqn:energy} 
 r(\mathbf{x}) = \sum_{i=1}^n \left\{ d(\mathbf{x}_i)^{-2}d(\mathbf{x})^{-2} \|\mathbf{x}_i-\mathbf{x}\|^{-4p} \right\}. 
\end{equation}
The new point $\mathbf{x}_{n+1}$ shall minimize $r(\mathbf{x}_{n+1})$ where $\mathbf{x}_{n+1}\in [0,1]^p$. 
\item Stop when a certain number of points are added or $\inf_{\mathbf{x}} r(\mathbf{x})$ goes above a value.  
\end{enumerate}
In this algorithm, high $d(\mathbf{x})$ implies relatively low output values and high $\|\mathbf{x}_i-\mathbf{x}\|$ implies good interpoint distance. 
As a result, the added runs tend to locate in low outcome regions and not too close to any completed run. 

It can be seen that the two algorithms are very similar to each other. 
The primary difference between them, as well as many other adaptive design methods, lies in the criterion  to choose follow-up runs (e.g., $g(\mathbf{x})$ in (\ref{eqn:g}) and $r(\mathbf{x})$ in (\ref{eqn:energy})). 
The criterion is the key to the success of adaptive designs. 
It needs to generate denser points in more interesting regions while scattering points uniformly in local regions. 
Most adaptive design methods are greedy in assuming that the next run to be added is the last run. 
As a result, if many follow-up runs are added in a local area, these points have no space-filling property. 

Sliced rotated sphere packing designs provide a non-greedy approach for adaptive designs. 
For the objective of emulating nonstationary computer experiments, our first strategy 
has the following three steps: 
\begin{enumerate}
\item Generate $(D_1,\ldots,D_s)$, an $(A_p^*,A_p,\mathbf{B})$-based sliced rotated sphere packing design using the second algorithm in Section~\ref{sec:const:SRSPD}. 
Run experiments using $D_1$ and obtain the outputs. 
\item Fit Gaussian process emulators using completed runs and add follow-up runs one-by-one. 
Treat points in $\cup_{j=2}^s D_j$ as candidates for follow-up runs. 
The new point $\mathbf{x}_{n+1}$ shall maximize $g(\mathbf{x}_{n+1})$ in (\ref{eqn:g}) where $\mathbf{x}_{n+1}\in \cup_{j=2}^s D_j$. 
\item Stop when a certain number of points are added or the $\sup_{\mathbf{x}} g(\mathbf{x})$ goes below a value.  
\end{enumerate}

Instead of searching for the entire $[0,1]^p$ to find an $\mathbf{x}$ that maximizes $g(\mathbf{x})$, we propose to search over a short list of candidate points. 
Apart from the apparent advantage of reduced computation, the new strategy ensures space-filling properties when multiple follow-up runs are added in a local area. 
For a local area with $m$ initial runs, up to roughly $pm$ follow-up runs can be added while preserving the $p^{1/2} (p+1)^{(-1-p)/(2p)} {n_1}^{-1/p}$ separation distance among all initial and follow-up runs, 
where $n_1$ is the number of initial runs. 

In the above strategy, we use the same criterion, namely $g(\mathbf{x})$, and the same stopping rules. 
Because baby points have the same interpoint distance, using $g(\mathbf{x})$ is equivalent to using $e(\mathbf{x})$ in (\ref{eqn:e}) alone. 
We now propose a simpler criterion. 
For $i=1,\ldots,n_1$, let 
\begin{equation}\label{eqn:te}
 \tilde e(\mathbf{x}) = \left[ \left\{ \sum_{i=1}^{n_1} \left( \hat f_{-i}(\mathbf{x}) - \hat f(\mathbf{x}) \right)^2 \right\}/n_1 \right]^{1/2},
\end{equation}
where $\hat f(\mathbf{x})$ is the predicted outcome from emulating all initial runs and $\hat f_{-i}(\mathbf{x})$ is the predicted outcome without using the $i$th run. 
For any baby point, let $\tilde e(\mathbf{x})$ be defined as the average $\tilde e$ value of its parents. 
In some rare cases, none of the parents of a baby point is located in $[0,1]^p$. 
Such baby points are assigned with highest $\tilde e(\mathbf{x})$. 
The $\tilde e(\mathbf{x})$ criterion is a further simplification from $e(\mathbf{x})$; using this criterion, we do not need to refit Gaussian process models after new runs completed. 

Below we summarize design strategies introduced for the emulation objective: 
\begin{description}
\item[MmLH] Use a non-adaptive maximin distance Latin hypercube design. 
\item[MmLH-CV] Use a maximin distance Latin hypercube design for $n_1$ initial runs; 
use the cross-validation based criterion $g(\mathbf{x})$ in (\ref{eqn:g}) to add follow-up runs. 
\item[SRSPD-CV] Use sliced rotated sphere packing design with $n_1$ adult points; 
use the cross-validation based criterion $g(\mathbf{x})$ in (\ref{eqn:g}) to add follow-up runs. 
\item[SRSPD-CV2] Use sliced rotated sphere packing design with $n_1$ adult points; 
use a modified error function $\tilde e(\mathbf{x})$ in (\ref{eqn:te}) to add follow-up runs. 
\end{description}
We compare these methods numerically on average prediction error from Gaussian process emulation over 10000 independently and uniformly sampled testing locations, assuming $n_1=13$ and altogether $n\geq 13$ runs are used. 
For each method, the results are averaged based on 100 randomly generated designs. 
To add randomness into SRSPD-CV and SRSPD-CV2, we use $w=100$ and randomly generated $R$s for sliced rotated sphere packing designs, which is different from our general recommendation for $p=2$. 
The results are shown in Figure~\ref{fig:resPEns2}. 
%


\begin{figure}
\begin{center}
\includegraphics[width=10cm]{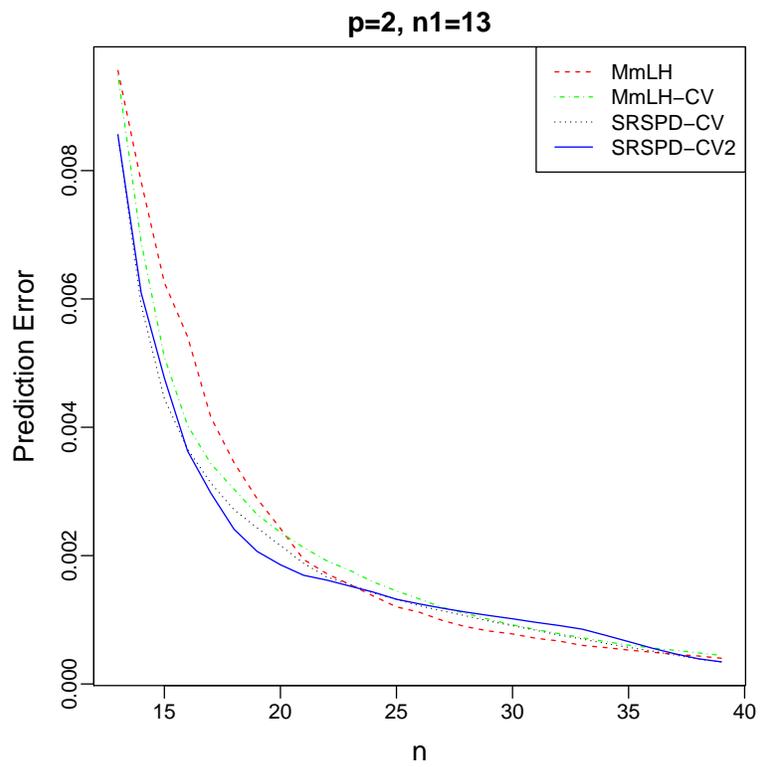}
\end{center}
\caption[width=10cm]{Mean prediction error for emulating the Franke's function. 
\label{fig:resPEns2}}
\end{figure}

Seen from the results, adaptive methods perform better than MmLH for $n\leq 20$. 
Clearly, adding no more than seven points in the bottom-left corner is better than scattering points uniformly in the design space. 
However, it is not beneficial to add more than seven follow-up runs since the bottom-left corner cannot contain too many points. 

The SRSPD-CV performs uniformly better than MmLH-CV. 
There are three possibly reasons for this. 
Firstly, follow-up runs of MmLH-CV may not be located in space-filling locations because the greedy one-at-a-time strategy cannot simultaneously control locations of multiply follow-up runs. 
Secondly, for MmLH-CV the balance between volatility and interpoint distance may not be ideal.
This may result in too much focus on either volatility or interpoint distance. 
Indeed, it is unknown if $g(\mathbf{x})$ in (\ref{eqn:g}) is inferior to $e(\mathbf{x}) \text{min}_{i=1}^n \|\mathbf{x}-\mathbf{x}_i\|^2$ or $e(\mathbf{x}) \text{min}_{i=1}^n \|\mathbf{x}-\mathbf{x}_i\|^{1/2}$. 
In contrast, for SRSPD-CV separation distance properties are ensured by the sliced lattice structure and $g(\mathbf{x})$ is used solely to measure volatility. 
Lastly, because it is computationally infeasible to compute $g(\mathbf{x})$ for every $\mathbf{x} \in [0,1]^p$, for MmLH-CV we only compute $g(\mathbf{x})$ on 5000 randomly generated positions as recommended by~\citet{Jin:2002}. 
Thus, the added runs may be suboptimal in $g(\mathbf{x})$. 

The SRSPD-CV2 performs better than SRSPD-CV for $17\leq n\leq 22$. 
This might because of the deficiency of $e(\mathbf{x})$. 
Clearly, $e(\mathbf{x})$ decreases as more points are added near $\mathbf{x}$. 
This may hinder adding more points in the left-bottom corner. 
In contrast, adult points from a sliced rotated sphere packing design have the same interpoint distance to other points, making $\tilde e(\mathbf{x})$ a fair measure on volatility. 
To sum it, when using sliced rotated sphere packing designs, much simpler criterion can be used; 
the parent-child relation may help in defining the criterion. 
Besides having better performance than MmLH and MmLH-CV, SRSPD-CV and SRSPD-CV2 take much less time.
The SRSPD-CV2 also allows follow-up runs to be added in parallel. 
%

We now return to the global minimization problem. 
The most popular method to the minimization problem is the EI algorithm below~\citep{EI}: 
\begin{enumerate}
\item Carry out $n_2$ initial runs that is uniformly distributed in the design space. 
\item Fit a Gaussian process emulator using completed runs and add follow-up runs one-by-one. 
Let the expected improvement of a new point $\mathbf{x}_{n+1}$ be 
\begin{equation} \label{eqn:EI}
\text{EI}(\mathbf{x}_{n+1}) = \text{E} \left( \left\{ \min_{i=1,\ldots,n} f(\mathbf{x}_i) - f(\mathbf{x}_{n+1})  \right\}^+ \mid \mathbf{x}_1,\ldots,\mathbf{x}_n,f(\mathbf{x}_1),\ldots,f(\mathbf{x}_n) \right), 
\end{equation}
where $\mathbf{x}_1,\ldots,\mathbf{x}_n$ are the evaluated runs, $z^+=z$ if $z\geq 0$ and $z^+=0$ if $z<0$. 
In the formula, $f(\mathbf{x}_{n+1})$ is random because the new run has not been carried out yet. 
\citet{EI} gave a deterministic formula to compute $\text{EI}(\mathbf{x}_{n+1})$ for any given $\mathbf{x}_{n+1}$. 
The new point $\mathbf{x}_{n+1}$ shall maximize $\text{EI}(\mathbf{x}_{n+1})$ where $\mathbf{x}_{n+1}\in [0,1]^p$. 
\item Stop when a certain number of points are added or $\sup_{\mathbf{x}} \text{EI}(\mathbf{x})$ goes below a value.  
\end{enumerate}

The minimization problem is different from the emulation problem. 
In order to ensure optimality, many points need to cluster around the minimum; 
these points cannot be space-filling. 
The EI criterion does exactly this. 
It was reported in~\citet{Joseph:energy} that sequential minimum energy designs do not work well for finding the single minimum. 
This is presumably because the energy function in (\ref{eqn:energy}) tends to spreads points away from each other. 
However, because sequential minimum energy designs give promising points for the response surface minimum, we develop an algorithm that combines the energy function with the EI criterion: 
 
\begin{enumerate}
\item Carry out $n_1$ initial runs that come from a maximin distance Latin hypercube design. 
\item Fit a Gaussian process emulator using completed runs and add $n_2-n_1$ follow-up runs one-by-one. 
The new point $\mathbf{x}_{n+1}$ shall minimize the energy function $r(\mathbf{x}_{n+1})$ in (\ref{eqn:energy}) where $\mathbf{x}_{n+1}\in [0,1]^p$. 
\item Fit a Gaussian process emulator using completed runs and add more follow-up runs one-by-one. 
The new point $\mathbf{x}_{n+1}$ shall maximize $\text{EI}(\mathbf{x}_{n+1})$ in (\ref{eqn:EI}) where $\mathbf{x}_{n+1}\in [0,1]^p$. 
\item Stop when a certain number of points are added or $\sup_{\mathbf{x}} \text{EI}(\mathbf{x})$ goes below a value.  
\end{enumerate}

The four-step algorithm above replaces the first step of the EI algorithm by sequential minimum energy designs with the same number of total runs. 
This adds adaptiveness to the $n_2$ initial runs.  
From our experience, the new algorithm generally outperforms the original EI algorithm. 
We further modify the algorithm using sliced rotated sphere packing designs. 
Because baby points of sliced rotated sphere packing designs have the same interpoint distance, 
it suffices to use the density function $d(\mathbf{x})$ in (\ref{eqn:d}) to replace the energy function $r(\mathbf{x}_{n+1})$. 
Also because baby points are located in the center of their parents, 
it suffices to use the average output value of parents as the predicted outcome of baby points. 
Clearly, this criterion is much simpler and model-free. 
Our proposed algorithm has the following three steps:

\begin{enumerate}
\item Generate $(D_1,\ldots,D_s)$, an $(A_p^*,A_p,\mathbf{B})$-based sliced rotated sphere packing design with $n_1$ adult points using the second algorithm in Section~\ref{sec:const:SRSPD}. 
Run experiments using $D_1$ and obtain the outputs. 
\item Fit a Gaussian process emulator using completed runs and add $n_2-n_1$ follow-up runs one-by-one. 
The new point $\mathbf{x}_{n+1}$ shall minimize the average output value among parents of $\mathbf{x}_{n+1}$, where $\mathbf{x}_{n+1} \in \cup_{j=2}^s D_j$. 
However, baby points with no parent are carried out with highest priority. 
\item Fit a Gaussian process emulator using completed runs and add more follow-up runs one-by-one. 
The new point $\mathbf{x}_{n+1}$ shall maximize $\text{EI}(\mathbf{x}_{n+1})$ in (\ref{eqn:EI}) where $\mathbf{x}_{n+1}\in [0,1]^p$. 
\item Stop when a certain number of points are added or $\sup_{\mathbf{x}} \text{EI}(\mathbf{x})$ goes below a value.  
\end{enumerate}

We compare the above-mentioned methods numerically: 
\begin{description}
\item[MmLH] The original EI algorithm using a maximin distance Latin hypercube design in the first step. 
\item[SMED] The four-step algorithm using the energy function. 
\item[SRSPD] The four-step algorithm using the average-parent-output criterion. 
\end{description}

As recommend in~\citet{EI}, we use $n_2=10p$ for all methods. 
Remark that for all three methods, the same EI criterion is used for the $(n_2+1)$th and subsequent runs. 
The difference lies in how the first $n_2$ runs are generated. 
For SMED and SRSPD, we set $n_1=13$ for $p=2$ and $n_1=5p$ for $p>2$. 
For a fair comparison, the stopping rule is set on the number of runs.  
We consider the four test functions that were used in~\citet{EI}, namely the Branin function, the Goldstein-Price function, the Hartmann 3 function and the Hartmann 6 function~\citep{test}. 
Their dimensions are 2, 2, 3 and 6, respectively. 
For each function with each strategy, we repeat the procedure for 100 times and compute the response value, namely the minimum output value among completed runs. 
To add randomness into the SRSPD strategy, we use $w=100$ and randomly generated $R$s for sliced rotated sphere packing designs. 
We depict the response value as a function of the number of completed runs in Figure~\ref{fig:EI}. 

\begin{figure}
\begin{center}
\includegraphics[width=13.6cm]{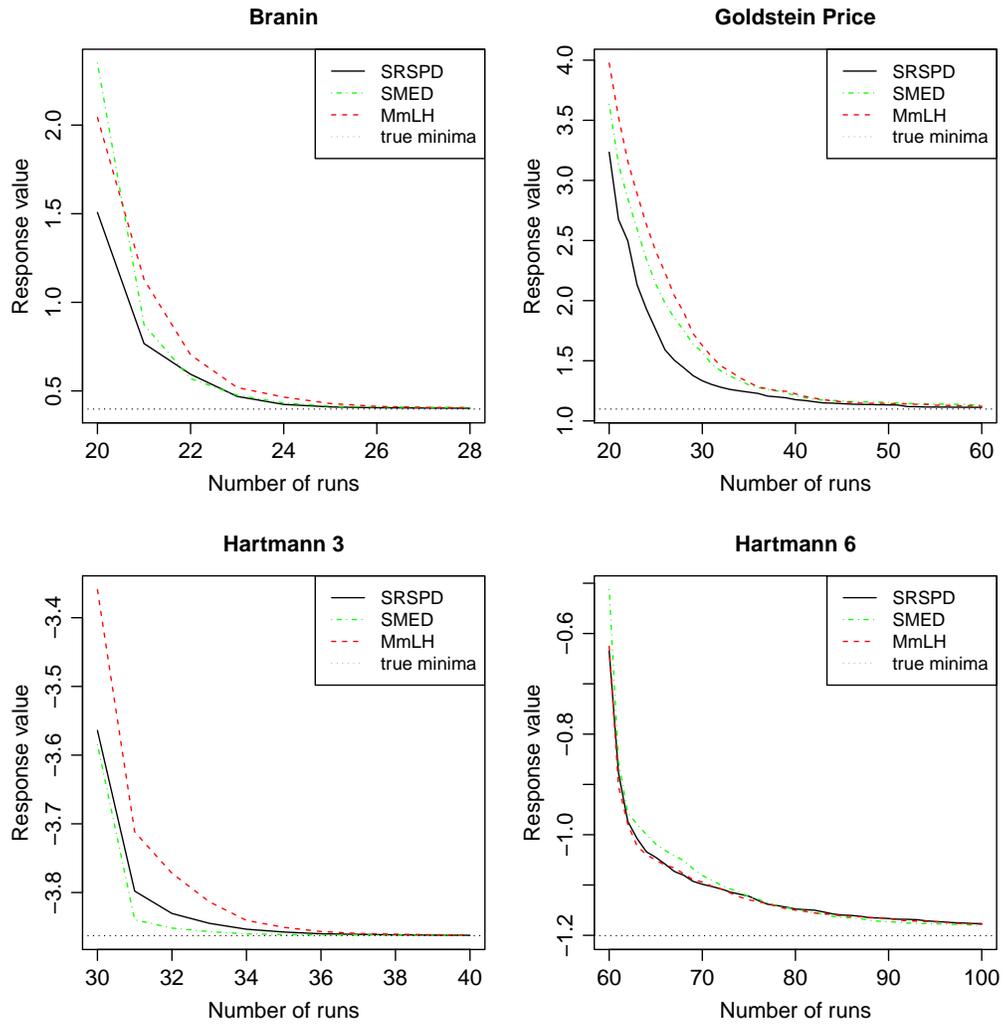}
\end{center}
\caption{Average response value as a function of the number of completed runs for algorithms using three types of initial designs. 
\label{fig:EI}}
\end{figure}

Seen from the results, both SRSPD and SMED perform well in finding promising points using the first $10p$ runs. 
In most cases, they continue to find good input sites earlier than MmLH. 
In particular, SRSPD is the best method for the Branin and Goldstein Price functions and one of the best methods for the Hartmann 6 function. 
Although not as good as SMED for Hartmann 3, SRSPD has the best overall performance. 
This clearly suggests that adaptive sliced rotated sphere packing designs are useful for the minimization problem.  
Similar to the emulation problem, 
the benefit may come from the distance properties of follow-up runs, the robustness of the simple average-parent-output criterion and the fact that we can obtain the exact optimum of the criterion. 
Besides, SRSPD takes less time and allows parallel computation in the second step. 

Although the most important component of adaptive designs is their criteria for choosing follow-up runs, 
our main focus here is not to provide new adaptive designs with new powerful criteria. 
Instead, our goal is to show the advantage of using a short list of candidate points for follow-up runs and that sliced rotated sphere packing designs are suitable under this strategy. 

We have shown that adaptive sliced rotated sphere packing designs can be used in combination with exact or simplified criterion that has been proposed before. 
For complex problems that no adaptive design criterion has been proposed, 
it should be easier to invent a criterion for our strategy than for usual adaptive designs. 
As discussed, a criterion for sliced rotated sphere packing designs only needs to measure how interesting positions are. 
In contrast, a criterion for usual adaptive designs need to strike a proper balance between more points in interesting regions and better distance properties. 
Furthermore, the lattice structure and the parent-child relation may help in developing fair criteria. 
The down side is that our strategy only allows points coming from two densities; it does not allow very dense points in a small region.

\section{Conclusions and discussion} \label{sec:conclusion}

In this paper, we propose a new class of sliced space-filling design called sliced rotated sphere packing designs. 
We also propose a space-filling type of sliced lattice, based on which sliced rotated sphere packing designs achieve good distance properties. 
Because of their delicate local structure, sliced rotated sphere packing designs are suitable as adaptive designs. 

The construction algorithms proposed in Section~\ref{sec:const:SRSPD} apply to any types of sliced lattices. 
Sometimes we should consider sliced lattices other than $(A_p^*,A_p,\mathbf{B})$. 
For example, to design computer experiments with one qualitative variable of $\tilde s$ levels and several quantitative variables, sliced lattices with exactly $\tilde s$ slices are desired. 
A future research problem is to construct sliced rotated sphere packing designs with flexible number of slices. 

We also propose a strategy to use sliced rotated sphere packing designs adaptively. 
The strategy requires a criterion for choosing follow-up runs that are suitable to the specific scientific goal. 
The criteria we have proposed for emulation and optimization problems may not be optimal. 
Further improvement by using more complex criteria are possible.
Our main focus is to corroborate the usefulness of the new strategy. 
Separate studies are needed to find the best algorithm for various applications 
such as 
finding an excursion set whose output is above a target value~\citep{Chevalier:2014} and   
estimating a percentile of the output distribution~\citep{Oakley:2004}.
We plan to work on these problems in the future. 

Similar to ordinary rotated sphere packing designs, a major restriction of sliced rotated sphere packing designs is on the number of dimensions. 
Although sliced rotated sphere packing designs are useful for $2\leq p\leq 6$, they are not suitable for high-dimensional problems.

\section*{Appendix}

\begin{proof}[Proof of Proposition~\ref{prp:slice}]
From (\ref{eqn:A_p^*}) and (\ref{eqn:A_p}), we have 
\[ \left\{2(p+1)\right\}^{1/2}\mathbf{M}_p = (\mathbf{I}_p+\mathbf{J}_p) \mathbf{M}_p^*; \quad \mathbf{M}_p^* = \left\{\mathbf{I}_p-\mathbf{J}_p/(p+1)\right\} \{2(p+1)\}^{1/2}\mathbf{M}_p. \]
Therefore, $(\mathbf{a}-\mathbf{u}_z)^T \mathbf{M}_p^* = (\mathbf{a}-\mathbf{u}_z)^T  \left\{\mathbf{I}_p-\mathbf{J}_p/(p+1)\right\} \{2(p+1)\}^{1/2}\mathbf{M}_p$. 
Because $(\mathbf{a}-\mathbf{u}_z)^T  \left\{\mathbf{I}_p-\mathbf{J}_p/(p+1)\right\}$ is an integer vector if and only if $\sum a_i \mod (p+1) = z$, 
$\mathbf{a}^T \mathbf{M}_p^* \in \mathbf{K} \oplus \mathbf{u}_z$ and $\cup_{k=0}^p ( \mathbf{K} \oplus \mathbf{u}_k ) =\mathbf{L}$. 
Since $\mathbf{u}_y$ and $\mathbf{u}_z$ do not belong to the same coset for $y,z\in\{0,\ldots,p\}$ and $y\neq z$, $\{  \mathbf{K} \oplus \mathbf{u}_0, \ldots,  \mathbf{K} \oplus \mathbf{u}_p \}$ is a partition of $\mathbf{L}$. 
\end{proof}

\begin{proof}[Proof of Proposition~\ref{prp:parent}]
(i) Consider an arbitrary adult point $\mathbf{a}^T \mathbf{M}_p^* \in  \mathbf{K}$. 
Let $\mathbf{a}=(a_1,\ldots,a_p)^T$, $\mathbf{b}=(b_1,\ldots,b_p)^T$ and $\mathbf{c}=(c_1,\ldots,c_p)^T=\mathbf{b}-\mathbf{a}$. 
Consider three cases for $\mathbf{c}$. 
Firstly, assume there exist $i,j \in \{1,\ldots,p\}$ such that $c_j-c_i\geq 2$. 
Let $ \mathbf{\tilde c} = \mathbf{c} +\mathbf{e}_i -\mathbf{e}_j$ where $\mathbf{e}_z$ is the $p$-vector with the $z$th element being one and other elements being zero. 
Then 
\[ \| \mathbf{c}^T \mathbf{M}_p^* \|^2 = \mathbf{c}^T \mathbf{M}_p^* (\mathbf{M}_p^*)^T \mathbf{c}
 = \mathbf{\tilde c}^T \mathbf{M}_p^* (\mathbf{M}_p^*)^T \mathbf{\tilde c} + 2 (\mathbf{c}-\mathbf{\tilde c})^T \mathbf{M}_p^* (\mathbf{M}_p^*)^T \mathbf{\tilde c} + (\mathbf{c}-\mathbf{\tilde c})^T \mathbf{M}_p^* (\mathbf{M}_p^*)^T (\mathbf{c}-\mathbf{\tilde c}) \]
\[ > \| \mathbf{\tilde c}^T \mathbf{M}_p^* \|^2 + 2 (\mathbf{c}-\mathbf{\tilde c})^T \left\{ \frac{p+1}{p} \mathbf{I}_p + \frac{-p-2+2\sqrt{p+1}}{p(\sqrt{p+1}-1)^2} \mathbf{J}_p \right\} \mathbf{\tilde c} \]
\[ = \| \mathbf{\tilde c}^T \mathbf{M}_p^* \|^2 + 2(p+1) (\mathbf{e}_j - \mathbf{e}_i)^T \mathbf{\tilde c} /p 
 \geq \| \mathbf{\tilde c}^T \mathbf{M}_p^* \|^2. \]
Because $(\mathbf{b}- \mathbf{\tilde c})^T \mathbf{M}_p^*$ is an adult point closer to $\mathbf{b}^T \mathbf{M}_p^*$ than $\mathbf{a}^T \mathbf{M}_p^*$, $\mathbf{a}^T \mathbf{M}_p^*$ is not a parent of $\mathbf{b}^T \mathbf{M}_p^*$. 

Secondly, assume $\min c_i \geq 1$ and there is a $j$ such that $c_j= 2$. 
Let $\mathbf{\tilde c} = \mathbf{c} -\mathbf{u}_p -\mathbf{e}_j$. 
Then 
\[ \| \mathbf{c}^T \mathbf{M}_p^* \|^2 > \| \mathbf{\tilde c}^T \mathbf{M}_p^* \|^2 + 2(p+1) \mathbf{e}_j^T \mathbf{\tilde c} /p \geq \| \mathbf{\tilde c}^T \mathbf{M}_p^* \|^2. \]
Because $(\mathbf{b}- \mathbf{\tilde c})^T \mathbf{M}_p^*$ is an adult point closer to $\mathbf{b}^T \mathbf{M}_p^*$ than $\mathbf{a}^T \mathbf{M}_p^*$, $\mathbf{a}^T \mathbf{M}_p^*$ is not a parent of $\mathbf{b}^T \mathbf{M}_p^*$. 

Similarly, the $\mathbf{a}^T \mathbf{M}_p^*$ with $\min c_i \leq -1$ and a $j$ such that $c_j= -2$ is not a parent of $\mathbf{b}^T \mathbf{M}_p^*$, either.  
Combining the three cases, 
a necessary condition for $\mathbf{a}^T \mathbf{M}_p^*$ being a parent of $\mathbf{b}^T \mathbf{M}_p^*$ is either $( c_i\in\{0,1\}, i=1,\ldots p, \sum c_i = z )$ or $( c_i \in \{0,-1\}, i=1,\ldots p, \sum c_i = z-(p+1) )$.
Because all $\mathbf{a}^T \mathbf{M}_p^*$ that satisfy the above conditions have the same distance to $\mathbf{b}^T \mathbf{M}_p^*$, we conclude that all of them are parents of $\mathbf{b}^T \mathbf{M}_p^*$. 

It is not hard to derive (ii) from (i). 
\end{proof}

\bibliographystyle{Chicago}
\bibliography{SlicedDesigns}

\begin{thebibliography}{}

\bibitem[\protect\citeauthoryear{Ai, Jiang, and Li}{Ai et~al.}{2014}]{Ai:2014}
Ai, M.~Y., B.~C. Jiang, and K.~Li (2014).
\newblock Construction of sliced space-filling designs based on balanced sliced
  orthogonal arrays.
\newblock {\em Statistica Sinica\/}~{\em 24\/}(4), 1685--1702.

\bibitem[\protect\citeauthoryear{Ba, Brenneman, and Myers}{Ba
  et~al.}{2015}]{Ba:2015}
Ba, S., W.~A. Brenneman, and W.~R. Myers (2015).
\newblock Optimal sliced latin hypercube designs.
\newblock {\em Technometrics\/}~{\em 57\/}(4), 479--487.

\bibitem[\protect\citeauthoryear{Chevalier, Ginsbourger, Bect, Vazquez,
  Picheny, and Richet}{Chevalier et~al.}{2014}]{Chevalier:2014}
Chevalier, C., D.~Ginsbourger, J.~Bect, E.~Vazquez, V.~Picheny, and Y.~Richet
  (2014).
\newblock Fast parallel kriging-based stepwise uncertainty reduction with
  application to the identification of an excursion set.
\newblock {\em Technometrics\/}~{\em 56\/}(4), 455--465.

\bibitem[\protect\citeauthoryear{Conway and Sloane}{Conway and
  Sloane}{1998}]{Conway:1998}
Conway, J.~H. and N.~J.~A. Sloane (1998).
\newblock {\em Sphere Packings, Lattices and Groups}.
\newblock New York: Springer.

\bibitem[\protect\citeauthoryear{Deng, Hung, and Lin}{Deng
  et~al.}{2015}]{Deng:2015}
Deng, X.~W., Y.~Hung, and C.~D. Lin (2015).
\newblock Design for computer experiments with qualitative and quantitative
  factors.
\newblock {\em Statistica Sinica\/}~{\em 25\/}(4), 1567--1581.

\bibitem[\protect\citeauthoryear{Deng, Lin, Liu, and Rowe}{Deng
  et~al.}{2016}]{Deng:2016}
Deng, X.~W., C.~D. Lin, K.-W. Liu, and R.~Rowe (2016).
\newblock Additive gaussian process for computer models with qualitative and
  quantitative factors.
\newblock {\em Technometrics\/}, to appear, DOI: 10.1080/00401706.2016.1211554.

\bibitem[\protect\citeauthoryear{Dixon and Szego}{Dixon and Szego}{1978}]{test}
Dixon, L. C.~W. and G.~P. Szego (1978).
\newblock The global optimization problem: an introduction.
\newblock {\em Towards global optimization\/}~{\em 2}, 1--15.

\bibitem[\protect\citeauthoryear{Gramacy and Lee}{Gramacy and
  Lee}{2008}]{Gramacy:2008}
Gramacy, R.~B. and H.~K.~H. Lee (2008).
\newblock Bayesian treed gaussian process models with an application to
  computer modeling.
\newblock {\em Journal of the American Statistical Association\/}~{\em 103},
  1119--1130.

\bibitem[\protect\citeauthoryear{Gramacy and Lee}{Gramacy and
  Lee}{2009}]{Gramacy:2009}
Gramacy, R.~B. and H.~K.~H. Lee (2009).
\newblock Adaptive design and analysis of supercomputer experiments.
\newblock {\em Technometrics\/}~{\em 51}, 130--145.

\bibitem[\protect\citeauthoryear{Haaland, Wang, and Maheshwari}{Haaland
  et~al.}{2017}]{Haaland:2017}
Haaland, B., W.~Wang, and V.~Maheshwari (2017).
\newblock A framework for controlling sources of inaccuracy in gaussian process
  emulation of deterministic computer experiments.
\newblock {\em SIAM/ASA Journal on Uncertainty Quantification\/}.
\newblock Under review, arXiv:1411.7049v3.

\bibitem[\protect\citeauthoryear{He}{He}{2016}]{RSPD}
He, X. (2016).
\newblock Rotated sphere packing designs.
\newblock {\em Journal of the American Statistical Association\/}, to appear,
  DOI: 10.1080/01621459.2016.1222289.

\bibitem[\protect\citeauthoryear{He}{He}{2017}]{ILmMD}
He, X. (2017).
\newblock Interleaved lattice-based minimax distance designs.
\newblock {\em Biometrika\/}, to appear, DOI: 10.1093/biomet/asx036.

\bibitem[\protect\citeauthoryear{He and Qian}{He and Qian}{2011}]{NOALH}
He, X. and P.~Z.~G. Qian (2011).
\newblock Nested orthogonal array-based {L}atin hypercube designs.
\newblock {\em Biometrika\/}~{\em 98}, 721--731.

\bibitem[\protect\citeauthoryear{Heitmann, Bingham, Lawrence, Bergner, Habib,
  Higdon, Pope, Biswas, Finkel, Frontiere, and Bhattacharya}{Heitmann
  et~al.}{2016}]{Heitmann:2016}
Heitmann, K., D.~Bingham, E.~Lawrence, S.~Bergner, S.~Habib, D.~Higdon,
  A.~Pope, R.~Biswas, H.~Finkel, N.~Frontiere, and S.~Bhattacharya (2016).
\newblock The mira–titan universe: Precision predictions for dark energy
  surveys.
\newblock {\em The Astrophysical Journal\/}~{\em 820\/}(2).

\bibitem[\protect\citeauthoryear{Huang, Yang, and Liu}{Huang
  et~al.}{2014}]{Huang:2014}
Huang, H.~Z., J.~F. Yang, and M.~Q. Liu (2014).
\newblock Construction of sliced (nearly) orthogonal latin hypercube designs.
\newblock {\em Journal of Complexity\/}~{\em 30\/}(3), 355--365.

\bibitem[\protect\citeauthoryear{Hwang, He, and Qian}{Hwang
  et~al.}{2016}]{SOALH}
Hwang, Y., X.~He, and P.~Z.~G. Qian (2016).
\newblock Sliced orthogonal array based {L}atin hypercube designs.
\newblock {\em Technometrics\/}~{\em 58\/}(1), 50--61.

\bibitem[\protect\citeauthoryear{Jin, Chen, and Sudjianto}{Jin
  et~al.}{2002}]{Jin:2002}
Jin, R., W.~Chen, and A.~Sudjianto (2002).
\newblock On sequential sampling for global metamodeling in engineering design.
\newblock In {\em Proceedings of ASME Design Engineering Technical Conferences
  And Computers and Information in Engineering Conference}, pp.\  539--548.

\bibitem[\protect\citeauthoryear{Johnson, Moore, and Ylvisaker}{Johnson
  et~al.}{1990}]{Johnson:1990}
Johnson, M.~E., L.~M. Moore, and D.~Ylvisaker (1990).
\newblock Minimax and maximin distance designs.
\newblock {\em Journal of Statistical Planning and Inference\/}~{\em 26},
  131--48.

\bibitem[\protect\citeauthoryear{Jones, Schonlau, and Welch}{Jones
  et~al.}{1998}]{EI}
Jones, D.~R., M.~Schonlau, and W.~J. Welch (1998).
\newblock Efficient global optimization of expensive black-box functions.
\newblock {\em Journal of Global Optimization\/}~{\em 13\/}(4), 455--492.

\bibitem[\protect\citeauthoryear{Joseph, Dasgupta, Tuo, and Wu}{Joseph
  et~al.}{2015}]{Joseph:energy}
Joseph, V.~R., T.~Dasgupta, R.~Tuo, and C.~F.~J. Wu (2015).
\newblock Sequential exploration of complex surfaces using minimum energy
  designs.
\newblock {\em Technometrics\/}~{\em 57\/}(1), 64--74.

\bibitem[\protect\citeauthoryear{Joseph, Gul, and Ba}{Joseph
  et~al.}{2015}]{Roshan:2015}
Joseph, V.~R., E.~Gul, and S.~Ba (2015).
\newblock Maximum projection designs for computer experiments.
\newblock {\em Biometrika\/}~{\em 102\/}(2), 371--380.

\bibitem[\protect\citeauthoryear{Kong, Ai, and Tsui}{Kong
  et~al.}{2016}]{Kong:2016}
Kong, X., M.~Ai, and K.~L. Tsui (2016).
\newblock Design for sequential follow-up experiments in computer emulations.
\newblock {\em Technometrics\/}, to appear, DOI: 10.1080/00401706.2016.1258010.

\bibitem[\protect\citeauthoryear{Liu and Liu}{Liu and Liu}{2015}]{Liu:2015}
Liu, H.~Y. and M.~Q. Liu (2015).
\newblock Column-orthogonal strong orthogonal arrays and sliced strong
  orthogonal arrays.
\newblock {\em Statistica Sinica\/}~{\em 25\/}(4), 1713--1734.

\bibitem[\protect\citeauthoryear{McKay, Conover, and Beckman}{McKay
  et~al.}{1979}]{McKay:1979}
McKay, M.~D., W.~J. Conover, and R.~J. Beckman (1979).
\newblock A comparison of three methods for selecting values of input variables
  in the analysis of output from a computer code.
\newblock {\em Technometrics\/}~{\em 21}, 239--245.

\bibitem[\protect\citeauthoryear{Morris and Mitchell}{Morris and
  Mitchell}{1995}]{Morris:1995}
Morris, M.~D. and T.~J. Mitchell (1995).
\newblock Exploratory designs for computational experiments.
\newblock {\em Journal of Statistical Planning and Inference\/}~{\em 43\/}(3),
  381--402.

\bibitem[\protect\citeauthoryear{Oakley}{Oakley}{2004}]{Oakley:2004}
Oakley, J. (2004).
\newblock Estimating percentiles of uncertain computer code outputs.
\newblock {\em Journal of the Royal Statistical Society: Series C\/}~{\em 53},
  83--93.

\bibitem[\protect\citeauthoryear{Qian}{Qian}{2012}]{Qian:2012}
Qian, P. (2012).
\newblock Sliced {L}atin hypercube designs.
\newblock {\em Journal of the American Statistical Association\/}~{\em
  107\/}(497), 393--399.

\bibitem[\protect\citeauthoryear{Qian}{Qian}{2009}]{Qian:2009}
Qian, P. Z.~G. (2009).
\newblock Nested {L}atin hypercube designs.
\newblock {\em Biometrika\/}~{\em 96}, 957--970.

\bibitem[\protect\citeauthoryear{Qian and Ai}{Qian and Ai}{2010}]{Qian:Ai:2010}
Qian, P. Z.~G. and M.~Ai (2010).
\newblock Nested lattice sampling: A new sampling scheme derived by randomizing
  nested orthogonal arrays.
\newblock {\em Journal of the American Statistical Association\/}~{\em 105},
  1147--1155.

\bibitem[\protect\citeauthoryear{Qian and Wu}{Qian and Wu}{2008}]{Qian:Wu:2008}
Qian, P. Z.~G. and C.~F.~J. Wu (2008).
\newblock Bayesian hierarchical modeling for integrating low-accuracy and
  high-accuracy experiments.
\newblock {\em Technometrics\/}~{\em 50}, 192--204.

\bibitem[\protect\citeauthoryear{Qian and Wu}{Qian and
  Wu}{2009}]{Qian:2009:SSFD}
Qian, P. Z.~G. and C.~F.~J. Wu (2009).
\newblock Sliced space-filling designs.
\newblock {\em Biometrika\/}~{\em 96\/}(4), 945--956.

\bibitem[\protect\citeauthoryear{Qian, Wu, and Wu}{Qian
  et~al.}{2008}]{Qian:Wu:Wu:2008}
Qian, P. Z.~G., H.~Wu, and C.~F.~J. Wu (2008).
\newblock Gaussian process models for computer experiments with qualitative and
  quantitative factors.
\newblock {\em Technometrics\/}~{\em 50}, 383--396.

\bibitem[\protect\citeauthoryear{Sacks, Welch, Mitchell, and Wynn}{Sacks
  et~al.}{1989}]{Sacks:1989}
Sacks, J., W.~J. Welch, T.~J. Mitchell, and H.~Wynn (1989).
\newblock Design and analysis of computer experiments.
\newblock {\em Statistical Science\/}~{\em 4}, 409--423.

\bibitem[\protect\citeauthoryear{Santner, Williams, and Notz}{Santner
  et~al.}{2003}]{Santner:book}
Santner, T.~J., B.~J. Williams, and W.~I. Notz (2003).
\newblock {\em The Design and Analysis of Computer Experiments}.
\newblock New York: Springer.

\bibitem[\protect\citeauthoryear{Sun, Liu, and Qian}{Sun
  et~al.}{2014}]{Sun:2014}
Sun, F.~S., M.~Q. Liu, and P.~Z.~G. Qian (2014).
\newblock On the construction of nested space-filling designs.
\newblock {\em Annals of Statistics\/}~{\em 42\/}(4), 1394--1425.

\bibitem[\protect\citeauthoryear{Xie, Xiong, Qian, and Wu}{Xie
  et~al.}{2014}]{Xie:2014}
Xie, H., S.~Xiong, P.~Z.~G. Qian, and C.~F.~J. Wu (2014).
\newblock General sliced latin hypercube designs.
\newblock {\em Statistica Sinica\/}~{\em 24\/}(3), 1239--1256.

\bibitem[\protect\citeauthoryear{Xu, Chen, and Qian}{Xu et~al.}{2015}]{Xu:2015}
Xu, J., J.~Chen, and P.~Z.~G. Qian (2015).
\newblock Sequentially refined {L}atin hypercube designs: Reusing every point.
\newblock {\em Journal of the American Statistical Association\/}~{\em
  110\/}(512), 1696--1706.

\bibitem[\protect\citeauthoryear{Yang, Lin, Qian, and Lin}{Yang
  et~al.}{2013}]{Yang:2013}
Yang, J.~F., C.~D. Lin, P.~Z.~G. Qian, and D.~K.~J. Lin (2013).
\newblock Construction of sliced orthogonal latin hypercube designs.
\newblock {\em Statistica Sinica\/}~{\em 23\/}(3), 1117--1130.

\bibitem[\protect\citeauthoryear{Zhang and Qian}{Zhang and
  Qian}{2013}]{Zhang:2013}
Zhang, Q. and P.~Z.~G. Qian (2013).
\newblock Designs for crossvalidating approximation models.
\newblock {\em Biometrika\/}~{\em 100\/}(4), 997--1004.

\bibitem[\protect\citeauthoryear{Zong}{Zong}{1999}]{Zong:1999}
Zong, C. (1999).
\newblock {\em Sphere Packings}.
\newblock New York: Springer.

\end{thebibliography}

\end{document}